\documentclass{amsart}
\usepackage[english]{babel}
\usepackage[T1]{fontenc}
\usepackage{lmodern} 
\usepackage[utf8]{inputenc}
\usepackage{microtype}
\usepackage[autostyle=true]{csquotes}
\usepackage{enumerate}
\usepackage{tikz}
\usetikzlibrary{cd}
\usepackage[colorlinks=false]{hyperref}
\usepackage{amsmath}
\usepackage{amssymb}
\usepackage{mathrsfs}
\usepackage{amsthm}
\usepackage{leftidx}
\usepackage{cleveref}
\usetikzlibrary{arrows,decorations.markings}
\usetikzlibrary{matrix}
\usetikzlibrary{graphs}
\usetikzlibrary{backgrounds}
\newcommand\Dynkindots{\hbox to 2em{.\hss.\hss.}}
\def\DynkinNodeSize{1.5mm}
\def\DynkinDoubleArrowLength{3.25mm}
\def\DynkinTripleArrowLength{3.5mm}
\tikzset{
  dnode/.style={
    circle,
    inner sep=0pt,
    minimum size=\DynkinNodeSize,
    fill=white,
    draw},
  bnode/.style={
    circle,
    inner sep=0pt,
    minimum size=\DynkinNodeSize,
    fill=black,
    draw},
  middlearrow/.style={
    decoration={markings,
      mark=at position 0.7 with
      {\draw (0:0mm) -- +(+160:\DynkinDoubleArrowLength); \draw (0:0mm) -- +(-160:\DynkinDoubleArrowLength);},
    },
    postaction={decorate}
  },
  triplemiddlearrow/.style={
    decoration={markings,
      mark=at position 0.7 with
      {\draw (0:0mm) -- +(+160:\DynkinTripleArrowLength); \draw (0:0mm) -- +(-160:\DynkinTripleArrowLength);},
    },
    postaction={decorate}
  },	
  sedge/.style={
  },
  dedge/.style={
    middlearrow,
    double distance=1mm,
  },
  tedge/.style={
    triplemiddlearrow,
    double distance=1.0mm+\pgflinewidth,
    postaction={draw},
  },
}
\title[On the values of unipotent characters in type \texorpdfstring{$E_7$}{E7}]{On the values of unipotent characters of finite Chevalley groups of type \texorpdfstring{$E_7$}{E7} in characteristic 2}
\author{Jonas Hetz}
\address{IAZ - Lehrstuhl für Algebra, Universität Stuttgart, Pfaffenwaldring 57, D--70569 Stuttgart, Germany}
\email{jonas.hetz@mathematik.uni-stuttgart.de}
\newtheorem{Thm}{Theorem}[section]
\newtheorem{Prop}[Thm]{Proposition}
\newtheorem{Lm}[Thm]{Lemma}

\theoremstyle{definition}

\newtheorem{Rem}[Thm]{Remark}
\newtheorem{Empty}[Thm]{}

\numberwithin{equation}{section}
\newcommand{\mf}{\mathbf}
\newcommand{\rom}[1]{\uppercase\expandafter{\romannumeral#1}}

\DeclareMathOperator{\Irr}{Irr}
\DeclareMathOperator{\Tr}{Tr}
\DeclareMathOperator{\Trace}{Trace}
\DeclareMathOperator{\ind}{ind}

\DeclareMathOperator{\CF}{CF}
\DeclareMathOperator{\rank}{rank}

\DeclareMathOperator{\End}{End}
\DeclareMathOperator{\Hom}{Hom}

\DeclareMathOperator{\IC}{IC}
\DeclareMathOperator{\Uch}{Uch}
\DeclareMathOperator{\supp}{supp}
\begin{document}
\begin{abstract}
Let $G$ be a finite group of Lie type. In order to determine the character table of $G$, Lusztig developed the theory of character sheaves. In this framework, one has to find the transformation between two bases for the space of class functions on $G$, one of them being the irreducible characters of $G$, the other one consisting of characteristic functions associated to character sheaves. In principle, this has been achieved by Lusztig and Shoji, but the underlying process involves some scalars which are still unknown in many cases. The problem of specifying these scalars can be reduced to considering cuspidal character sheaves. We will deal with the latter for the specific case where $G=E_7(q)$, and $q$ is a power of the bad prime $p=2$ for $E_7$.
\end{abstract}


\maketitle

\section{Introduction}\label{Intro}

Let $\mf G$ be a connected reductive algebraic group over the algebraic closure $k=\overline{\mathbb F}_p$ of the field with $p$ elements (for a prime $p$). Assume that $\mf G$ is defined over the finite subfield $\mathbb F_q$ of $k$, where $q$ is a power of $p$, so the $\mathbb F_q$-rational points on $\mf G$ constitute the corresponding finite group of Lie type $\mf G(q)=\mf G(\mathbb F_q)$. The theory of character sheaves, due to Lusztig \cite{LuCS1}-\cite{LuCS5}, provides a general procedure by which the problem of determining the character table of $\mf G(q)$ can be tackled. Character sheaves are certain geometric analogues to the irreducible representations of $\mf G(q)$, and they give rise to class functions on $\mf G(q)$, called characteristic functions. Lusztig (\cite[13.7]{Luchars}, \cite{LuCS5}) conjectured that any such characteristic function coincides up to multiplication by a non-zero scalar with an appropriate almost character of $\mf G(q)$, that is, an explicitly known linear combination of the irreducible characters. Then he provides an algorithm to compute the characteristic functions, at least in principle (\cite[\S 24]{LuCS5}). While Lusztig's conjecture is still open in general (see, e.g., \cite{LuDAC}), it has been proven by Shoji (\cite{Sh1}, \cite{Sh2}) under the assumption that the centre of $\mf G$ is connected. Even in this case, however, the scalars relating characteristic functions of character sheaves and almost characters need to be specified. Following \cite[\S 3]{Luvaluni}, this task can be reduced to considering cuspidal character sheaves, by means of an induction process for these objects which leads to a parametrisation of the character sheaves on $\mf G$, rather analogous to Harish-Chandra theory for irreducible characters of $\mf G(q)$.

It therefore seems reasonable to consider the individual quasisimple groups of Lie type separately. In those cases, \cite{LuIC} combined with \cite{LuCS1}-\cite{LuCS5}, see also \cite[\S 7]{Sh1}, \cite[\S 5]{Sh2}, provides a classification of cuspidal character sheaves. We will be concerned with unipotent character sheaves on $\mf G$, that is, the subset of the character sheaves on $\mf G$ corresponding to the unipotent characters of $\mf G(q)$ due to Deligne-Lusztig \cite{DL}.

For classical groups of split type, Shoji \cite{Shvaluniclass}, \cite{Shvaluniclass2} has shown that the characteristic functions of cuspidal unipotent character sheaves coincide with the corresponding almost characters. As far as exceptional groups are concerned, this statement is only partially verified. E.g., for groups of type $F_4$, it has been proven by Marcelo-Shinoda \cite{MaShi} for those cuspidal character sheaves whose support contains unipotent elements, in all characteristics. (In type $F_4$, any characteristic, every cuspidal character sheaf is unipotent.) The particular case $F_4$, $p=2$, is in fact complete, as the respective identities for not unipotently supported cuspidal character sheaves have been verified in \cite[\S 5]{Gvaluni}.

Moreover, by a result of Geck \cite[Proposition 3.4]{Gvaluni}, the computation of the scalars relating characteristic functions of unipotent character sheaves and unipotent almost characters can be reduced to the base case $p=q$. This is particularly useful in small characteristics, as one might hope to settle such cases via direct computations, e.g. by applying computer algebra methods. Several concrete examples are considered in \cite{Gvaluni}, such as the above-mentioned $F_4$, $p=2$, as well as the split adjoint group of type $E_6$ in characteristic $2$, where the problem is solved for all cuspidal character sheaves (there are two of them, and both are unipotent). Note that, in either of these cases, the character table of $\mf G(2)$ is known. However, it is also conjectured in \cite[6.6]{Gvaluni} that an analogous statement will hold for the adjoint groups of type $E_6$ in good characteristic, even though the character table of $\mf G(p)$ is not known here. Finally, \cite{HE6p3} deals with $E_6$, $p=3$ (both split and non-split).

In this paper, we shall investigate the case where $\mf G$ is a simple Chevalley group of type $E_7$ over $k=\overline{\mathbb F}_2$, defined over $\mathbb F_q$, $q$ a power of $2$. There are two cuspidal character sheaves on $\mf G$, and both of them are unipotent \cite[Proposition 20.3]{LuCS4}. After singling out a specific normalisation of the characteristic functions of these character sheaves, we show that they coincide with the corresponding almost characters of $\mf G(q)$. In order to achieve this, we will, as in \cite{HE6p3}, consider the Hecke algebra associated to $\mf G(q)$ and apply a formula in \cite{GCH} which gives a relation between the characters of this Hecke algebra and the principal series unipotent characters of $\mf G(q)$. However, we will need to take into account some more theoretical machinery than in \cite{HE6p3}. This identification of the characteristic functions of cuspidal character sheaves with the respective almost characters immediately determines the values at the regular unipotent elements in $\mf G(q)$ of the two cuspidal unipotent characters and of the two principal series unipotent characters corresponding to the two irreducible characters of the Weyl group of degree $512$. More information on character values at unipotent elements is provided by Shoji's and Lusztig's algorithms \cite[\S 5]{ShArc}, \cite[\S24]{LuCS5} for the computation of Green functions (as defined in \cite{DL}). While these algorithms in general involve some unknown scalars, they have been determined for our case recently in \cite{GcompGreen}.

\begin{Empty}\label{Not} \textbf{Notation.}
Whenever $\mf G$ is a connected reductive algebraic group over $k=\overline{\mathbb F}_p$, we assume to have fixed a Borel subgroup $\mf B\subseteq\mf G$ as well as a maximal torus $\mf T$ of $\mf G$ which is contained in $\mf B$, and we denote by $\mf W:=N_{\mf G}(\mf T)/{\mf T}$ the Weyl group of $\mf G$ with respect to $\mf T$. Let $\mf U:=R_u(\mf B)$ be the unipotent radical of $\mf B$. Then $\mf B$ is the semidirect product of $\mf U$ and $\mf T$ (with $\mf U$ being normal in $\mf B$).

If, in addition, $\mf G$ is defined over a finite subfield $\mathbb F_q\subseteq k$ (where $q$ is a power of $p$) with Frobenius map $F\colon\mf G\rightarrow\mf G$, we write $\mf G(q)=\mf G^F=\{g\in\mf G\mid F(g)=g\}$ for the corresponding finite group of Lie type. In this case, unless otherwise stated, we will also tacitly assume that both $\mf B$ and $\mf T$ are $F$-stable. Thus $F$ induces an automorphism of $\mf W$ which we shall denote again by $F$.

We will be concerned with characters of $\mf G^F$ in characteristic $0$. As usual in the ordinary representation theory of finite groups of Lie type, we consider representations and characters over $\overline{\mathbb Q}_\ell$, an algebraic closure of the field of $\ell$-adic numbers, for a fixed prime $\ell$ different from $p$. Thus given a finite group $\Gamma$, let $\CF(\Gamma)$ be the set of class functions $\Gamma\rightarrow\overline{\mathbb Q}_\ell$ and denote by
\[\langle f, f'\rangle_{\Gamma}:=|\Gamma|^{-1}\sum_{g\in\Gamma}f(g)\overline{f'(g)}\quad (\text{for }f, f'\in\CF(\Gamma))\]
the standard scalar product on $\CF(\Gamma)$, where bar denotes a field automorphism of $\overline{\mathbb Q}_\ell$ which maps roots of unity to their inverses. Let $\Irr(\Gamma)\subseteq\CF(\Gamma)$ be the subset of irreducible characters of $\Gamma$. They form an orthonormal basis of $\CF(\Gamma)$ with respect to this scalar product. If $\Gamma=\mf G^F$, we denote by $\Uch(\mf G^F)\subseteq\Irr(\mf G^F)$ the subset of unipotent characters, that is, those $\rho\in\Irr(\mf G^F)$ which satisfy $\langle\rho,R_w\rangle\neq0$ for some $w\in\mf W$. Here, $R_w$ is the virtual character defined by Deligne and Lusztig in \cite{DL}.
\end{Empty}

\section{Character sheaves}\label{CS}

Let $\mf G$ be a connected reductive algebraic group over $k=\overline{\mathbb F}_p$. We begin by very briefly introducing some notions of the theory of character sheaves (only those which are relevant for our purposes), following \cite{LuCS1}-\cite{LuCS5}, \cite{LuIntroCS}, \cite{LuclCS}.
\begin{Empty}\label{CSwoFrob}
Denote by $\mathscr D\mf G$ the bounded derived category of constructible $\overline{\mathbb Q}_\ell$-sheaves on $\mf G$, and by $\mathscr M\mf G$ the full subcategory of perverse sheaves on $\mf G$ in the sense of \cite{BBD}. $\mathscr M\mf G$ is an abelian category. The character sheaves on $\mf G$ are defined (\cite[2.10]{LuCS1}) as certain simple objects of $\mathscr M\mf G$ which are equivariant for the conjugation action of $\mf G$ on itself. We fix a set of representatives $\hat{\mf G}$ for the isomorphism classes of character sheaves on $\mf G$. 

An important subset of the character sheaves on $\mf G$ are the so-called unipotent character sheaves, defined as follows. For $w\in\mf W$, let $K_w:=K_w^{{\mathscr L}_0}\in\mathscr D\mf G$ be as defined in \cite[2.4]{LuCS1} with respect to the constant local system $\mathscr L_0=\overline{\mathbb Q}_\ell$ on $\mf T$. An element of $\hat{\mf G}$ is called a unipotent character sheaf if it is a constituent of a perverse cohomology sheaf $\leftidx{^p}{\!H}{^i}(K_w)$ for some $i\in\mathbb Z$, $w\in\mf W$. We denote by
\[{\hat{\mf G}}^\mathrm{un}:=\{A\in\hat{\mf G}\mid A\text{ unipotent}\}\subseteq\hat{\mf G}\]
the unipotent character sheaves in $\hat{\mf G}$. We also set (see \cite[14.10]{LuCS3})
\[R_\phi^{\mathscr L_0}:=\frac1{|\mf W|}\sum_{w\in\mf W}\phi(w)\sum_{i\in\mathbb Z}{{(-1)}^{i+\dim\mf G}}\, {\leftidx{^p}{\!H}{^i}(K_w)}\quad\text{for }\phi\in\Irr(\mf W),\]
an element of the subgroup of the Grothendieck group of $\mathscr M\mf G$ spanned by the character sheaves. Denote by $(\;:\;)$ the symmetric $\overline{\mathbb Q}_\ell$-bilinear pairing on this subgroup such that for any two character sheaves $A$, $A'$ on $\mf G$ we have
\[(A:A')=\begin{cases}1 \quad\text{if}\quad A\cong A'\\ 0 \quad\text{if}\quad A\ncong A'.\end{cases}\]
Finally, there is the notion of cuspidal character sheaves on a connected reductive group, see \cite[3.10]{LuCS1}. By \cite[Proposition 3.12]{LuCS1}, any cuspidal character sheaf on $\mf G$ has the form $\IC(\overline\Sigma,\mathscr E)[\dim\Sigma]$, where \enquote{IC} stands for the intersection cohomology complex due to Deligne-Goresky-MacPherson (\cite{GMP}, \cite{BBD}), $\Sigma$ is the inverse image of a conjugacy class under the natural map $\mf G\rightarrow\mf G/{\mf Z(\mf G)^\circ}$, and $\mathscr E$ is a $\mf G$-equivariant irreducible $\overline{\mathbb Q}_\ell$-local system on $\Sigma$, such that $(\Sigma,\mathscr E)$ is a cuspidal pair for $\mf G$ in the sense of \cite{LuIC}. We set
\[{\hat{\mf G}}^\circ:=\{A\in\hat{\mf G}\mid A\text{ cuspidal character sheaf}\}\quad\text{and}\quad\hat{\mf G}^{\circ,\mathrm{un}}:=\hat{\mf G}^\mathrm{un}\cap\hat{\mf G}^\circ.\]
This gives rise to an inductive description of (unipotent) character sheaves, as follows. Let $\mf L\subseteq\mf G$ be a Levi complement of some parabolic subgroup $\mf P\subseteq\mf G$. To each complex $K\in\mathscr M\mf L$ which is equivariant for the conjugation action of $\mf L$ on itself is associated an induced complex $\ind_{\mf L\subseteq\mf P}^{\mf G}(K)\in\mathscr D\mf G$, see \cite[4.1]{LuCS1}. If $A_0\in\hat{\mf L}^\circ$, the complex $\ind_{\mf L\subseteq\mf P}^{\mf G}(A_0)$ is a semisimple perverse sheaf on $\mf G$, all its simple direct summands are character sheaves on $\mf G$, and any character sheaf on $\mf G$ is a simple direct summand of $\ind_{\mf L\subseteq\mf P}^{\mf G}(A_0)$ for some $\mf L\subseteq\mf P$ as above, and some $A_0\in\hat{\mf L}^\circ$. (In fact, for the latter statement it is enough to consider standard Levi subgroups $\mf L$ of standard parabolic subgroups $\mf P$. We will use this in \Cref{ParHC} below.) Moreover, $A_0$ is unipotent if and only if some (or equivalently, every) simple summand of $\ind_{\mf L\subseteq\mf P}^{\mf G}(A_0)$ is unipotent (\cite[4.4, 4.8]{LuCS1}). In this way, the character sheaves on $\mf G$ correspond to irreducible modules for the endomorphism algebras (in $\mathscr M\mf G$) of various induced complexes as above, using the respective Hom functor (see \cite[\S 10]{LuCS2}).
\end{Empty}
\begin{Prop}[{\cite[\S 10]{LuCS2}}, {\cite[3.8]{LuRCS}}, {\cite[5.9]{Sh1}}]\label{LusztigParCSh}
Let $\mf L\subseteq\mf G$ be a Levi complement of the parabolic subgroup $\mf P\subseteq\mf G$. Assume that $\hat{\mf L}^{\circ,\mathrm{un}}$ is non-empty and fix some $A_0\in\hat{\mf L}^{\circ,\mathrm{un}}$. Then the endomorphism algebra $\End_{\mathscr M\mf G}\bigl(\ind_{\mf L\subseteq\mf P}^{\mf G}(A_0)\bigr)$ is isomorphic to the group algebra $\overline{\mathbb Q}_\ell[W_{\mf G}(\mf L)]$, where $W_{\mf G}(\mf L):=N_{\mf G}(\mf L)/{\mf L}$ is the relative Weyl group of $\mf L$ in $\mf G$. Hence, we obtain a bijection
\[\Irr(W_{\mf G}(\mf L))\xrightarrow{1-1}\left\{A\in\hat{\mf G}\mid A\text{ is a simple direct summand of }\ind_{\mf L\subseteq\mf P}^{\mf G}(A_0)\right\}.\]
\end{Prop}

\begin{Empty}\label{CSwithFrob}
From now until the end of this section, we assume that the connected reductive group $\mf G$ is defined over $\mathbb F_q$ (where $q$ is a power of $p$), with corresponding Frobenius map $F\colon\mf G\rightarrow\mf G$. For simplicity, we shall also assume that $F$ induces the identity on $\mf W$.

Given $A\in\mathscr D\mf G$, let $F^\ast A\in\mathscr D\mf G$ be the inverse image of $A$ under the Frobenius map $F$. Suppose that $F^\ast A$ is isomorphic to $A$ and choose an isomorphism $\varphi\colon F^\ast A\xrightarrow{\sim}A$. Then $\varphi$ induces linear maps $\varphi_{i,g}\colon \mathscr H_g^i(A)\rightarrow\mathscr H_g^i(A)$ for $i\in\mathbb Z$ and $g\in\mf G^F$, where $\mathscr H_g^i(A)$ is the stalk at $g$ of the $i$th cohomology sheaf of $A$, a $\overline{\mathbb Q}_\ell$-vector space of finite dimension. The characteristic function $\chi_{A,\varphi}\colon\mf G^F\rightarrow\overline{\mathbb Q}_\ell$ associated with $A$ (and $\varphi$) is defined by (see \cite[8.4]{LuCS2})
\[\chi_{A,\varphi}(g):=\sum_{i\in\mathbb Z}(-1)^i\Trace(\varphi_{i,g},\mathscr H_g^i(A))\quad\text{for }g\in\mf G^F.\]
This is well-defined since only finitely many of the $\mathscr H_g^i(A)$ ($i\in\mathbb Z$) are non-zero. In the case where $A$ is a $\mf G$-equivariant perverse sheaf, we have $\chi_{A,\varphi}\in\CF(\mf G^F)$. If, in addition, $A$ is a simple object of $\mathscr M\mf G$, then $\varphi$ (and, hence, $\chi_{A,\varphi}$) is unique up to a non-zero scalar. Denote by 
\[{\hat{\mf G}}^F:=\{A\in\hat{\mf G}\mid F^\ast A\cong A\}\subseteq\hat{\mf G}\]
the $F$-stable character sheaves in $\hat{\mf G}$. The following is one of the main results of \cite{LuCS5}, which we can now state without any restriction on the characteristic in view of \cite{LuclCS}.
\end{Empty}
\begin{Thm}[Lusztig {\cite[\S 25]{LuCS5}, \cite[3.10]{LuclCS}}]\label{LusztigONB}
In the setting of \ref{CSwithFrob}, there exist isomorphisms $\varphi_A\colon F^\ast A\xrightarrow{\sim}A$ (for $A\in{\hat{\mf G}}^F$) which satisfy the following conditions:
\begin{enumerate}
\item[(a)] The values of the characteristic functions $\chi_{A,\varphi_A}$ are cyclotomic integers;
\item[(b)] $\{\chi_{A,\varphi_A}\mid A\in{\hat{\mf G}}^F\}$ is an orthonormal basis of $\left(\CF(\mf G^F),\langle\;,\;\rangle_{\mf G^F}\right)$.
\end{enumerate}
\end{Thm}
The required properties for the isomorphisms $\varphi_A\colon F^\ast A\xrightarrow{\sim}A$ (formulated in \cite[25.1]{LuCS5}, \cite[13.8]{LuCS3}) determine $\varphi_A$ up to multiplication by a root of unity. Sometimes, for $A\in{\hat{\mf G}}^F$, it will be convenient to just write $\chi_A$ for a characteristic function associated to the character sheaf $A$, without referring to a specific isomorphism $\varphi_A\colon F^\ast A\xrightarrow{\sim}A$. Whenever we do this, we tacitly assume that $\chi_A=\chi_{A,\varphi_A}$ for a(ny) chosen isomorphism $\varphi_A\colon F^\ast A\xrightarrow{\sim}A$ which satisfies these properties.

\section{Parametrisation of unipotent characters and unipotent character sheaves}\label{ParUchUSh}

$\mf G$ is a connected reductive algebraic group over $k=\overline{\mathbb F}_p$, defined over $\mathbb F_q$ with Frobenius map $F\colon\mf G\rightarrow\mf G$. In this section, we assume in addition that the centre of $\mf G$ is connected and that the map induced by $F$ on $\mf W$ is the identity. This assumption ensures that all unipotent character sheaves on $\mf G$ are $F$-stable. In \Cref{ParHC} below, we will further restrict to the case where $\mf G$ is the simple adjoint group of type $E_7$. There are two natural ways to give \enquote{parallel} parametrisations of unipotent characters of $\mf G^F$ and of ($F$-stable) unipotent character sheaves on $\mf G$. We will see however that they are compatible in the case of $E_7$.
\begin{Empty}\label{ParFamUch}
According to \cite[Main Theorem 4.23]{Luchars}, $\Uch(\mf G^F)$ can be classified in terms of the following data, which only depend on $\mf W$, and not on $p$ or $q$. First, Lusztig describes a partition of the irreducible characters of $\mf W$ into families. To each such family $\mathscr F\subseteq\Irr(\mf W)$ is associated a certain finite group $\mathscr G=\mathscr G_{\mathscr F}$ and a set $\mathscr M(\mathscr G)$ consisting of all pairs $(g,\sigma)$, $g\in\mathscr G$, $\sigma\in\Irr(C_{\mathscr G}(g))$, modulo the relation $(g,\sigma)\sim(hgh^{-1},\leftidx{^h}\sigma)$ for $h\in\mathscr G$, where $\leftidx{^h}\sigma$ is the irreducible character of $C_{\mathscr G}(hgh^{-1})=hC_{\mathscr G}(g)h^{-1}$ given by composing $\sigma$ with conjugation by $h^{-1}$. $\mathscr M(\mathscr G)$ is equipped with a pairing $\{\;,\;\}\colon\mathscr M(\mathscr G)\times\mathscr M(\mathscr G)\rightarrow\overline{\mathbb Q}_\ell$, defined by
\[\big\{(g,\sigma),(h,\tau)\big\}:=\frac1{|C_{\mathscr G}(g)|}\frac1{|C_{\mathscr G}(h)|}\sum_{\substack{x\in\mathscr G \\ gxhx^{-1}=xhx^{-1}g}}\tau(x^{-1}g^{-1}x)\sigma(xhx^{-1}),\]
for $(g,\sigma), (h,\tau)\in\mathscr M(\mathscr G)$. Given a family $\mathscr F\subseteq\Irr(\mf W)$, let $\mathscr F\hookrightarrow\mathscr M(\mathscr G_{\mathscr F})$ be the embedding defined in \cite[4.14]{Luchars}. Setting 
\begin{equation}\label{ParSet}
\overline X(\mf W):=\coprod_{\substack{\mathscr F\subseteq\Irr(\mf W) \\ \text{family}}}\mathscr M(\mathscr G_\mathscr F),
\end{equation}
we thus have an embedding
\begin{equation}\label{IrrWembed}
\Irr(\mf W)\hookrightarrow\overline X(\mf W),\quad \phi\mapsto x_\phi.
\end{equation}
The pairing $\{\;,\;\}$ is extended to $\overline X(\mf W)$ such that $\{(g,\sigma),(h,\tau)\}:=0$ whenever $(g,\sigma), (h,\tau)\in\overline X(\mf W)$ are not in the same $\mathscr M(\mathscr G_\mathscr F)$. Then $\overline X(\mf W)$ parametrises $\Uch(\mf G^F)$ as follows. For $\phi\in\Irr(\mf W)$, let
\[R_\phi:=\frac 1{|\mf W|}\sum_{w\in\mf W}\phi(w)R_w.\]
Then there is a bijection
\begin{equation}\label{ParUch}
\Uch(\mf G^F)\xrightarrow{\sim}\overline X(\mf W),\quad\rho\mapsto\overline x_\rho,
\end{equation}
such that for any $\rho\in\Uch(\mf G^F)$ and any $\phi\in\Irr(\mf W)$, we have
\begin{equation}\label{MultUch}
\langle\rho,R_\phi\rangle_{\mf G^F}=\Delta(\overline x_\rho)\{\overline x_\rho,x_\phi\}.
\end{equation}
Here, $\Delta(\overline x_\rho)\in\{\pm1\}$ is a sign attached to $\rho\in\Uch(\mf G^F)$, see \cite[4.21]{Luchars}. The bijection \eqref{ParUch} is not uniquely determined by this property. In the case where $\mf G$ is of type $E_7$, we will make a definite choice in \Cref{Parcompatible} below. Next, for $x\in\overline X(\mf W)$, a unipotent \enquote{almost character} $R_x\in\CF(\mf G^F)$ is defined by
\[R_x:=\sum_{\rho\in\Uch(\mf G^F)}\{\overline x_\rho,x\}\Delta(\overline x_\rho)\rho.\]
We then have $R_\phi=R_{x_\phi}$ for any $\phi\in\Irr(\mf W)$. Since the \enquote{Fourier matrix} $\Upsilon:=(\{x,x'\})_{x,x'\in\overline X(\mf W)}$ is hermitian and $\Upsilon^2$ is the identity matrix \cite[\S 4]{LuUniE8}, we obtain
\[\langle R_x,R_{x'}\rangle_{\mf G^F}=\begin{cases}1&\text{ if }x=x',\\0&\text{ if }x\neq x'\end{cases}\quad(\text{for }x,x'\in\overline X(\mf W)).\]
It follows that
\begin{equation}\label{Fourier}
\rho=\Delta(\overline x_\rho)\sum_{x\in\overline X(\mf W)}\overline{{\{\overline x_\rho,x\}}}R_x\quad\text{for }\rho\in\Uch(\mf G^F).
\end{equation}
In particular, knowing the values of the unipotent characters of $\mf G^F$ is equivalent to knowing the values of the $R_x$ for $x\in\overline X(\mf W)$.
\end{Empty}

\begin{Empty}\label{ParFamUcsh}
On the other hand, by \cite[Theorem 23.1]{LuCS5} combined with the cleanness results in \cite{LuclCS}, the set $\overline X(\mf W)$ in \eqref{ParSet} also serves as a parameter set for ${\hat{\mf G}}^\mathrm{un}$. Namely, there is a bijection
\begin{equation}\label{ParUchSh}
\overline X(\mf W)\xrightarrow{\sim}{\hat{\mf G}}^\mathrm{un},\quad x\mapsto A_x,
\end{equation}
such that for any $x\in\overline X(\mf W)$ and any $\phi\in\Irr(\mf W)$, we have
\begin{equation}\label{MultUcsh}
\bigl(A_x:R_\phi^{\mathscr L_0}\bigr)=\hat\varepsilon_{A_x}\{x,x_\phi\}.
\end{equation}
Here, $\hat\varepsilon_K=(-1)^{\dim\mf G-\dim\supp K}$ for any complex $K\in\mathscr D\mf G$, and $\supp K$ is the Zariski closure of $\{g\in\mf G\mid \mathscr H^i_g K\neq0\text{ for some }i\in\mathbb Z\}\subseteq\mf G$. Once again, this property does not uniquely determine the bijection \eqref{ParUchSh}. For $\mf G$ of type $E_7$, we will fix a choice in \Cref{Parcompatible}.

With these notions we can formulate the following theorem of Shoji, which verifies Lusztig's conjecture under the assumption that $\mf G$ has connected centre. As mentioned in \cite[2.7]{Gvaluni}, this holds without any conditions on $p$, $q$, since the cleanness of cuspidal character sheaves is established in full generality (\cite{LuclCS}).
\begin{Thm}[Shoji {\cite[Theorems 3.2, 4.1]{Sh2}}]\label{Shoji}
Let $p$ be a prime, $q$ a power of $p$, $\mf G$ a connected reductive group over $\overline{\mathbb F}_p$, defined over $\mathbb F_q$ with corresponding Frobenius map $F\colon\mf G\rightarrow\mf G$. Assume that $\mf Z(\mf G)$ is connected and $\mf G/{\mf Z(\mf G)}$ is simple. Then ${\hat{\mf G}}^\mathrm{un}\subseteq{\hat{\mf G}}^F$ and for any $x\in\overline X(\mf W)$, $R_x$ and $\chi_{A_x}$ coincide up to a non-zero scalar.
\end{Thm}
\end{Empty}
\begin{Rem}
In general, any irreducible character of $\mf G^F$ has a corresponding almost character, and the set of all these almost characters is an orthonormal basis for $\CF(\mf G^F)$ (\cite[\S 4]{Luchars}). Shoji formulates and proves \Cref{Shoji} for the characteristic function of any $F$-stable character sheaf and the respective almost character. Determining the character table of $\mf G^F$ can thus be reformulated to solving the following two problems (see \cite{LucompIrr}):
\begin{enumerate}
\item[(a)] Computing the values of the characteristic functions of $F$-stable character sheaves;
\item[(b)] Determining the scalar relating the characteristic function of any $F$-stable character sheaf with the corresponding almost character.
\end{enumerate}
(b) can be reduced to considering cuspidal ($F$-stable) character sheaves, via the induction process for cuspidal character sheaves mentioned in \Cref{CS}, see \cite[\S 3]{Luvaluni}. On the other hand, Lusztig \cite[\S 24]{LuCS5} shows that (a) can be reduced to the computation of \enquote{generalised Green functions} and he provides an algorithm to compute those functions (refining an earlier algorithm of Shoji in \cite{ShArc}), although it involves some unspecified scalars which are still not known in all cases. We will briefly describe this algorithm in the next section.
\end{Rem}

\begin{Empty}\label{ParHC}
We assume here that $\mf G$ is the simple adjoint group of type $E_7$ over the field $k=\overline{\mathbb F}_p$, defined over $\mathbb F_q$, $q$ a power of $p$, with corresponding Frobenius map $F\colon\mf G\rightarrow\mf G$. With respect to this $\mf G$, $\mf T\subseteq\mf B\subseteq\mf G$ and $\mf W$ are as in the introduction. These choices determine the set of roots $R\subseteq X(\mf T)=\Hom(\mf T,k^\times)$, as well as the positive roots $R^+\subseteq R$ and the simple roots $\Pi:=\{\alpha_1,\ldots,\alpha_7\}\subseteq R^+$. We choose the order of $\alpha_1,\ldots, \alpha_7$ in such a way that the Dynkin diagram of $E_7$ is as follows:
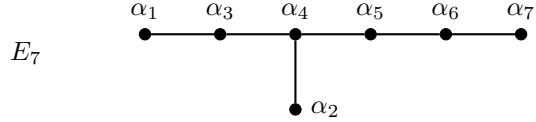
\begin{center}
\begin{tikzpicture}
    \draw (-1.25,-0.25) node[anchor=east]  {$E_7$};

    \node[bnode,label=above:$\alpha_1$] 		(1) at (0,0) 	{};
    \node[bnode,label=right:$\alpha_2$] 		(2) at (2,-1) 	{};
    \node[bnode,label=above:$\alpha_3$] 		(3) at (1,0) 	{};
    \node[bnode,label=above:$\alpha_4$] 		(4) at (2,0) 	{};
    \node[bnode,label=above:$\alpha_5$] 		(5) at (3,0) 	{};
    \node[bnode,label=above:$\alpha_6$] 		(6) at (4,0) 	{};
    \node[bnode,label=above:$\alpha_7$]			(7) at (5,0) 	{};

    \path 	(1) edge[thick, sedge] (3)
          	(3) edge[thick, sedge] (4)
          	(4) edge[thick, sedge] (5)
			(4)	edge[thick, sedge] (2)
          	(5) edge[thick, sedge] (6)
          	(6) edge[thick, sedge] (7);
\end{tikzpicture}
\end{center}
$\mf W$ can be viewed as a Coxeter group of type $E_7$ with Coxeter generators $\mf S=\{s_1,\ldots, s_7\}$, where $s_i$ denotes the reflection in the hyperplane orthogonal to $\alpha_i$. We will describe the parametrisations of $\Uch(\mf G^F)$ and $\hat{\mf G}^{\mathrm{un}}$ in terms of Harish-Chandra series, following Lusztig \cite[\S 3]{LuRCS}. (Note that all the results in loc. cit. hold whenever ${\mf G}/{\mf Z(\mf G)}$ is simple or $\{1\}$, and we will implicitly use this below as far as subgroups of $\mf G$ corresponding to irreducible subgraphs of the Dynkin diagram of type $E_7$ are concerned.)

To any Coxeter system $(W, S)$ is attached (\cite[3.1]{LuRCS}) a certain finite (possibly empty) set $\mathfrak S_{W}^\circ$. If $W$ is irreducible, $\mathfrak S_{W}^\circ$ is contained in the set of roots of unity in $\overline{\mathbb Q}_\ell$. Let
\[\mathfrak S_{W}:=\left\{(J,\epsilon,\zeta)\mid J\subseteq S,\, \epsilon\in\Irr(W^{S/J}),\, \zeta\in\mathfrak S_{W_J}^\circ\right\},\]
where $W_J:=\langle J\rangle\subseteq W$ and $W^{S/J}$ is the subgroup of $W$ generated by the involutions $\sigma_s:=w_0^{J\cup\{s\}}w_0^J=w_0^Jw_0^{J\cup\{s\}}$ for $s\in S\backslash J$ (for $J'\subseteq S$, we denote by $w_0^{J'}$ the longest element in $W_{J'}$). Then both $\left(W_{J},J\right)$ and $\left(W^{S/J},\{\sigma_s\mid s\in S\backslash J\}\right)$ are Coxeter systems, see \cite{LuCoxFrob}, \cite{Lurepchev}.
\begin{equation}\label{sqrtp}
\begin{split}
&\text{From now on we shall fix, once and for all, a square root } \sqrt p \text{ of } p \text{ in }\overline{\mathbb Q}_\ell. \\
&\text{Furthermore, when }q=p^f (f\geqslant1), \text{ we write } \sqrt q:=\sqrt p^f.
\end{split}
\end{equation}
Now consider $\rho\in\Uch(\mf G^F)$. By \cite[3.25, 3.26]{Lurepchev}, there exist $J\subseteq\mf S$ and a cuspidal unipotent character $\rho_0\in\Uch(\mf L_J^F)$ such that $\langle\rho,R_{\mf L_J}^{\mf G}(\rho_0)\rangle_{\mf G^F}\neq0$. (For $J\subseteq\mf S$, we denote by $\mf L_J$ the unique Levi complement of the standard parabolic subgroup $\mf P_J:=\mf B\mf W_J\mf B\subseteq\mf G$ such that $\mf T\subseteq\mf L_J$, and $R_{\mf L_J}^{\mf G}$ is Harish-Chandra induction.) After having chosen $\sqrt p$ in \eqref{sqrtp}, there is a natural isomorphism between the endomorphism algebra $\End_{\overline{\mathbb Q}_\ell\mf G^F}\bigl(R_{\mf L_J}^{\mf G}(\rho_0)\bigr)$ and the group algebra $\overline{\mathbb Q}{_\ell}\bigl[\mf W^{\mf S/J}\bigr]$ (in analogy to \ref{LusztigParCSh}; note that $\mf W^{\mf S/J}$ is isomorphic to $W_\mf G(\mf L_J)$). Hence, among the constituents of $R_{\mf L_J}^{\mf G}(\rho_0)$, $\rho$ is naturally parametrised by some $\epsilon\in\Irr(\mf W^{\mf S/J})$ (see also \cite[Corollary 8.7]{Luchars}). Moreover, denote by $\lambda_\rho$ the eigenvalue of Frobenius of $\rho$ (as defined in \cite[\S 11]{Luchars}). Then we have $\lambda_\rho=\lambda_{\rho_0}\in\mathfrak S_{\mf W_J}^\circ$ and the assignment $\rho\mapsto(J,\epsilon,\lambda_\rho)$ gives a well-defined bijection
\[\Uch(\mf G^F)\xrightarrow{\sim}\mathfrak S_{\mf W}.\]
In particular, the cuspidal unipotent characters correspond to the elements of $\mathfrak S_\mf W^\circ$. On the other hand, let $A\in\hat{\mf G}^{\mathrm{un}}$. By \ref{CSwoFrob}, \ref{LusztigParCSh}, there exist $J\subseteq\mf S$, $A_0\in\hat{\mf L}_J^{\circ,\mathrm{un}}$, such that $A$ is a simple summand of $\ind_{\mf L_J\subseteq\mf P_J}^{\mf G}(A_0)$, and $A$ is parametrised by some $\epsilon\in\Irr(\mf W^{\mf S/J})$. Furthermore, to $A$ is associated a root of unity $\lambda_A$ by means of a certain \enquote{twisting operator} defined via Shintani descent, see \cite[1.16, Theorem 3.3]{Sh1}. Another definition of $\lambda_A$ is given in \cite[3.6]{LuRCS}. Then we have $\lambda_A=\lambda_{A_0}$ (see \cite[3.6]{Sh1}), and this is an element of $\mathfrak S_{\mf W_J}^\circ$. The assignment $(J,\epsilon,\lambda_A)\mapsto A$ gives a well-defined bijection
\[\mathfrak S_{\mf W}\xrightarrow{\sim}\hat{\mf G}^{\mathrm{un}}.\]
In particular, the unipotent cuspidal character sheaves correspond to the elements of $\mathfrak S_\mf W^\circ$.
\end{Empty}

\begin{Rem}\label{Parcompatible}
We keep the setting of \Cref{ParHC}. In particular, $\mf Z(\mf G)$ is connected and $F$ acts trivially on $\mf W$. Recall from \Cref{ParFamUch}, \Cref{ParFamUcsh}, that the bijections $\Uch(\mf G^F)\xrightarrow{\sim}\overline X(\mf W)$ and $\overline X(\mf W)\xrightarrow{\sim}\hat{\mf G}^{\mathrm{un}}$ are not uniquely determined by the conditions \eqref{MultUch}, \eqref{MultUcsh}, respectively. We will now make a definite choice following \cite[11.2]{Luchars}, see also \cite[\S 6]{DMParam}.

Given $\rho\in\Uch(\mf G^F)$, the family $\mathscr F\subseteq\Irr(\mf W)$ for which $\mathscr M(\mathscr G_\mathscr F)$ contains $\overline x_\rho$ is independent of a chosen bijection $\Uch(\mf G^F)\xrightarrow{\sim}\overline X(\mf W)$ in \eqref{ParUch} satisfying \eqref{MultUch}. Consider an element $x=(g,\sigma)\in\mathscr M(\mathscr G_\mathscr F)\subseteq\overline X(\mf W)$. If $\mathscr F$ contains the irreducible characters of $\mf W$ of degree $512$, and if $x$ is not in the image of the embedding $\mathscr F\hookrightarrow\mathscr M(\mathscr G_\mathscr F)$, we set $\tilde\lambda_x:=\zeta\sigma(g)/\sigma(1)$, where $\zeta$ is a primitive fourth root of unity in $\overline{\mathbb Q}_\ell$. In any other case, let $\tilde\lambda_x:=\sigma(g)/\sigma(1)$. Then it turns out that the bijection \eqref{ParUch} can be chosen in such a way that $\lambda_\rho=\tilde\lambda_{\overline x_\rho}$ for any $\rho\in\Uch(\mf G^F)$, and it is almost uniquely determined by this condition along with \eqref{MultUch}, the only ambiguity arises from the two elements of $\Irr(\mf W)$ of degree $512$. This can be removed as follows. We have a natural embedding $\Irr(\mf W)\hookrightarrow\mathfrak S_\mf W$, $\phi\mapsto(\emptyset,\phi,1)$. Then we require that the embeddings $\Irr(\mf W)\hookrightarrow\mathfrak S_\mf W\xrightarrow{\sim}\Uch(\mf G^F)$ and $\Irr(\mf W)\hookrightarrow\overline X(\mf W)\xrightarrow{\sim}\Uch(\mf G^F)$ coincide. This does not interfere with our choices above, in view of \cite[Proposition 12.6]{Luchars}, and it uniquely specifies the bijection $\Uch(\mf G^F)\xrightarrow{\sim}\overline X(\mf W)$.

Now let $A\in\hat{\mf G}^{\mathrm{un}}$ and let $(J,\epsilon,\lambda_A)\in\mathfrak S_\mf W$, $\rho\in\Uch(\mf G^F)$ be the elements associated to $A$ in \Cref{ParHC}. By \cite[3.10]{LuRCS}, we have $\langle\rho,R_\phi\rangle_{\mf G^F}=\bigl(A:R_\phi^{\mathscr L_0}\bigr)$ for any $\phi\in\Irr(\mf W)$. Moreover, $\Delta(\overline x_\rho)=\hat\varepsilon_A$, so \eqref{MultUch} implies $\bigl(A:R_\phi^{\mathscr L_0}\bigr)=\hat\varepsilon_A\{\overline x_\rho,x_\phi\}$ for any $\phi\in\Irr(\mf W)$. Hence, the assignment $\overline x_\rho\mapsto A$ gives rise to a bijection $\overline X(\mf W)\xrightarrow{\sim}\hat{\mf G}^{\mathrm{un}}$ satisfying \eqref{MultUcsh}.
Thus we can and will from now on always assume that the bijections $\Uch(\mf G^F)\xrightarrow{\sim}\overline X(\mf W)\xrightarrow{\sim}\hat{\mf G}^\mathrm{un}$ in \Cref{ParFamUch}, \Cref{ParFamUcsh} are chosen in this way. In particular, the following diagram is commutative.
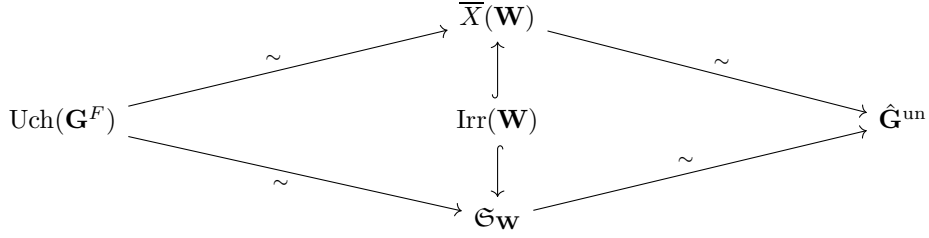
\begin{center}
\begin{tikzcd}[column sep=12em, row sep=2em]
& \overline X(\mf W)\arrow[dr, "\sim"] &
\\
\Uch(\mf G^F)\arrow[ur, "\sim"]\arrow[dr, "\sim"'] & \Irr(\mf W)\arrow[d, hook]\arrow[u,hook] & \hat{\mf G}^{\mathrm{un}}
\\
& \mathfrak S_{\mf W}\arrow[ur, "\sim"] &
\end{tikzcd}
\end{center}
\end{Rem}

\section{The generalised Springer correspondence}\label{GenSpring}

In this section, $\mf G$ may be assumed to be an arbitrary connected reductive algebraic group over $\overline{\mathbb F}_p$, as in the introduction. We will focus on those character sheaves whose support contains unipotent elements. In \cite[\S 24]{LuCS5}, Lusztig provides an algorithm for computing the values of their characteristic functions in principle, using intersection cohomology complexes \cite{LuIC}. Further references for the following are \cite[1.1-1.3]{ShGreen1}, \cite[\S 3, \S 4]{Tayvaluni}.
\begin{Empty}\label{PreAlg}
Denote by $\mathscr N_{\mf G}$ the set of all pairs $(\mathscr O,\mathscr E)$ where $\mathscr O$ is a unipotent conjugacy class in $\mf G$ and $\mathscr E$ is the isomorphism class of an irreducible $\mf G$-equivariant $\overline{\mathbb Q}_\ell$-local system on $\mathscr O$. For a given $\mathscr O$, these local systems on $\mathscr O$ naturally correspond to irreducible characters of the group $A(u):=C_{\mf G}(u)/C_{\mf G}^\circ(u)$, where $u$ is any fixed element of $\mathscr O$. To $(\mathscr O,\mathscr E)\in\mathscr N_{\mf G}$ is associated (\cite[Theorem 6.5]{LuIC}) a triple $(\mf L,\mathscr O_0,\mathscr E_0)$ consisting of a closed subgroup $\mf L\subseteq\mf G$ which is the Levi complement of some parabolic subgroup of $\mf G$, a unipotent conjugacy class $\mathscr O_0$ in $\mf L$ and an $\mf L$-equivariant $\overline{\mathbb Q}_\ell$-local system $\mathscr E_0$ on $\mathscr O_0$ such that $(\mf Z(\mf L)^\circ\times\mathscr O_0,1\boxtimes\mathscr E_0)$ is a cuspidal pair for $\mf L$ in the sense of \cite[2.4]{LuIC}. Denoting by $\mathscr M_\mf G$ the set of all those triples up to $\mf G$-conjugacy, the corresponding map $\tau\colon\mathscr N_\mf G\rightarrow\mathscr M_\mf G$ is surjective, so we have a decomposition
\[\mathscr N_{\mf G}=\coprod_{(\mf L,\mathscr O_0,\mathscr E_0)\in\mathscr M_{\mf G}}\tau^{-1}(\mf L,\mathscr O_0,\mathscr E_0)\]
of $\mathscr N_{\mf G}$ into non-empty subsets. The elements in a given $\tau^{-1}(\mf L,\mathscr O_0,\mathscr E_0)$ can be distinguished as follows. The triple $(\mf L,\mathscr O_0,\mathscr E_0)$ gives rise to a semisimple perverse sheaf $K_{(\mf L,\mathscr O_0,\mathscr E_0)}\in\mathscr M\mf G$ defined as a certain intersection cohomology complex in \cite[\S 4]{LuIC}, see also \cite[\S 8]{LuCS2}. Then for any $\iota=(\mathscr O_\iota,\mathscr E_\iota)\in\tau^{-1}(\mf L,\mathscr O_0,\mathscr E_0)$ there is (up to isomorphism) exactly one simple direct summand $A_\iota$ of $K_{(\mf L,\mathscr O_0,\mathscr E_0)}$ which satisfies
\begin{equation}\label{CSIC}
A_\iota|_{\mf G_\mathrm{uni}}\cong\IC\bigl(\overline{\mathscr O_\iota},\mathscr E_\iota\bigr)\bigl[\dim\mathscr O_\iota+\dim\mf Z(\mf L)^\circ\bigr]
\end{equation}
($\mf G_{\mathrm{uni}}\subseteq\mf G$ is the unipotent subvariety of $\mf G$), and each simple direct summand of $K_{(\mf L,\mathscr O_0,\mathscr E_0)}$ arises in this way from some $\iota\in\tau^{-1}(\mf L, \mathscr O_0, \mathscr E_0)$, see \cite[6.5]{LuIC}, \cite[24.1]{LuCS5}. On the other hand, the endomorphism algebra $\mathscr A:=\End_{\mathscr M\mf G}\bigl(K_{(\mf L,\mathscr O_0,\mathscr E_0)}\bigr)$ is isomorphic to $\overline{\mathbb Q}_\ell[W_\mf G(\mf L)]$ (\cite[Theorem 9.2]{LuIC}). Thus the isomorphism classes of simple direct summands of $K_{(\mf L,\mathscr O_0,\mathscr E_0)}$ are naturally parametrised by $\Irr(W_\mf G(\mf L))$. Hence, we obtain a bijection
\begin{equation}\label{Springer}
\mathscr N_{\mf G}\cong\coprod_{(\mf L,\mathscr O_0,\mathscr E_0)\in\mathscr M_{\mf G}}\Irr(W_\mf G(\mf L)),
\end{equation}
which is called the generalised Springer correspondence. The problem of explicitly determining this correspondence can be reduced to the case where $\mf G$ is almost simple of simply-connected type, and has been solved by Lusztig and Spaltenstein, except for a few minor cases in type $E_6$ with $p\neq3$, and in type $E_8$ with $p=3$ (\cite{LuIC}, \cite{LuSp}, \cite{Sp}, \cite{LugenSpringer}). In particular, the generalised Springer correspondence is completely determined in the case of $E_7$, see \cite{Sp}.
\end{Empty}

\begin{Empty}\label{PreAlgF}
Let us now bring the $\mathbb F_q$-rational structure on $\mf G$ into the picture, see \cite[\S 24]{LuCS5}. So we assume that $\mf G$ is defined over $\mathbb F_q$, with Frobenius map $F\colon\mf G\rightarrow\mf G$. $F$ acts naturally on both $\mathscr N_\mf G$ and $\mathscr M_\mf G$ via
\[\mathscr N_\mf G\rightarrow\mathscr N_\mf G,\;(\mathscr O,\mathscr E)\mapsto(F^{-1}(\mathscr O),F^\ast\mathscr E),\]
and
\[\mathscr M_\mf G\rightarrow\mathscr M_\mf G,\;(\mf L,\mathscr O_0,\mathscr E_0)\mapsto(F^{-1}(\mf L),F^{-1}(\mathscr O_0),F^\ast\mathscr E_0),\]
respectively. If $\mf L$ is $F$-stable, it is also clear that $F$ acts on $W_\mf G(\mf L)$ and, hence, on $\Irr(W_\mf G(\mf L))$, by
\[\Irr(W_\mf G(\mf L))\rightarrow\Irr(W_\mf G(\mf L)),\;\epsilon\mapsto\epsilon\circ F.\]
Let $\mathscr N_\mf G^F\subseteq\mathscr N_\mf G$, $\mathscr M_\mf G^F\subseteq\mathscr M_\mf G$, ${\Irr(W_\mf G(\mf L))}^F\subseteq\Irr(W_\mf G(\mf L))$ be the respective subsets of fixed points under these actions (where, in terms of the local systems, this is only meant up to isomorphism). The generalised Springer correspondence \eqref{Springer} is compatible with the above actions, thus it induces a bijection
\begin{equation}\label{SpringerF}
\mathscr N_{\mf G}^F\cong\coprod_{(\mf L,\mathscr O_0,\mathscr E_0)\in\mathscr M_{\mf G}^F}{\Irr(W_\mf G(\mf L))}^F.
\end{equation}
For any $\iota=(\mathscr O,\mathscr E)\in\mathscr N_{\mf G}^F$, the complex $A_\iota$ is $F$-stable. More precisely, let $\tau(\iota)=(\mf L, \mathscr O_0, \mathscr E_0)\in\mathscr M_{\mf G}^F$. We may choose an isomorphism $F^\ast\mathscr E_0\xrightarrow{\sim}\mathscr E_0$ which induces a map of finite order at the stalk of $\mathscr E_0$ at any element of $\mathscr O_0^F$. Such a choice induces an isomorphism $F^\ast K_{(\mf L, \mathscr O_0, \mathscr E_0)}\xrightarrow{\sim}K_{(\mf L, \mathscr O_0, \mathscr E_0)}$, which in turn gives rise to an isomorphism $\varphi_{A_\iota}\colon F^\ast A_\iota\xrightarrow{\sim}A_\iota$, as described in \cite[24.2]{LuCS5}. Once $\varphi_{A_\iota}$ is fixed, it determines an isomorphism $\psi_\iota\colon F^\ast\mathscr E\xrightarrow{\sim}\mathscr E$ (via \cite[(24.2.2)]{LuCS5}), and for any $g\in\mathscr O^F$, the induced map $\psi_{\iota,g}\colon\mathscr E_g\rightarrow\mathscr E_g$ on the stalk of $\mathscr E$ at $g$ is of finite order. Now define two functions $Y_\iota, X_\iota\colon\mf G^F_{\mathrm{uni}}\rightarrow\overline{\mathbb Q}_\ell$ by
\[Y_\iota(g):=\begin{cases}
\Trace(\psi_{\iota,g},\mathscr E_g)\quad&\text{if }g\in \mathscr O^F, \\
 \hfil0\quad&\text{if }g\notin \mathscr O^F
\end{cases}\]
and
\[X_\iota(g):=(-1)^{\dim \mathscr O+\dim{\mf Z(\mf L)}^\circ}q^{-d_\iota}\chi_{A_\iota,\varphi_{A_\iota}}(g)\]
for $g\in\mf G^F_{\mathrm{uni}}$, where $d_\iota=\frac12(\dim\supp A_\iota-\dim \mathscr O-\dim{\mf Z(\mf L)}^\circ)$. Then both $Y_\iota$ and $X_\iota$ are invariant under the conjugation action of $\mf G^F$ on $\mf G^F_{\mathrm{uni}}$.
\end{Empty}
\begin{Thm}[Lusztig {\cite[\S 24]{LuCS5}}]\label{Alg}
In the setting of \Cref{PreAlgF}, the following hold.
\begin{enumerate}
\item[(a)] The functions $Y_\iota$, $\iota\in\mathscr N_{\mf G}^F$, form a basis for the vector space consisting of all functions $\mf G^F_{\mathrm{uni}}\rightarrow\overline{\mathbb Q}_\ell$ which are invariant under $\mf G^F$-conjugacy.
\item[(b)] There is a system of equations
\[X_\iota=\sum_{\iota'\in\mathscr N_{\mf G}^F}p_{\iota',\iota}Y_{\iota'}\quad(\iota\in\mathscr N_{\mf G}^F),\]
for some uniquely determined $p_{\iota',\iota}\in\mathbb Z$.
\end{enumerate}
\end{Thm}
Note that the restrictions (23.0.1) on the characteristic of $k$ can be removed since the cleanness of cuspidal character sheaves has been established in complete generality \cite{LuclCS}.
\begin{Rem}\label{RemarkAlg}
In the setting of \Cref{Alg}, if we define a total order $\leqslant$ on $\mathscr N_\mf G^F$ in such a way that for $\iota=(\mathscr O,\mathscr E)$, $\iota'=(\mathscr O',\mathscr E')\in\mathscr N_\mf G^F$ we have $\iota'\leqslant\iota$ whenever $\mathscr O'\subseteq\overline{\mathscr O}$ (the latter condition defines a partial order on the set of unipotent classes of $\mf G$), the matrix $(p_{\iota',\iota})_{\iota',\iota\in\mathscr N_\mf G^F}$ has upper unitriangular shape, as follows immediately from the definitions. Lusztig \cite[24.4]{LuCS5} provides an algorithm for computing this matrix $(p_{\iota',\iota})$, which entirely relies on combinatorial data. This algorithm is implemented in {\sffamily {CHEVIE}} \cite{MiChv} and is accessible via the functions \texttt{UnipotentClasses} and \texttt{ICCTable}. Hence, the computation of the characteristic functions $\chi_{A_\iota,\varphi_{A_\iota}}$ ($\iota\in\mathscr N_{\mf G}^F$) at unipotent elements is reduced to the computation of the functions $Y_\iota$, $\iota\in\mathscr N_\mf G^F$. It should be noted, however, that the latter is still a non-trivial task since it is difficult to describe the isomorphisms $\varphi_{A_\iota}$ (and thus the $\psi_\iota$) explicitly. This has been accomplished for classical groups (in any characteristic) by Shoji, \cite{ShGreen1}, \cite{ShGreen2}, \cite{ShGreen3}. As far as exceptional groups are concerned, this problem is not yet solved.
\end{Rem}

\begin{Rem}\label{ICindA0}
In the setting of \Cref{PreAlg}, let us fix a triple $(\mf L,\mathscr O_0,\mathscr E_0)\in\mathscr M_\mf G$ and consider the complex $K_{(\mf L,\mathscr O_0,\mathscr E_0)}\in\mathscr M\mf G$. By \cite[4.5]{LuIC}, we have a canonical isomorphism
\[K_{(\mf L,\mathscr O_0,\mathscr E_0)}\cong\ind_{\mf L\subseteq\mf P}^{\mf G}\bigl(\IC(\overline\Sigma_0,\mathscr E_0)[\dim\Sigma_0]\bigr),\text{ where }\Sigma_0=\mf Z(\mf L)^\circ\mathscr O_0.\]
Here, $A_0:=\IC(\overline\Sigma_0,\mathscr E_0)[\dim\Sigma_0]$ is a cuspidal character sheaf on $\mf L$. In particular, since the definition of $K_{(\mf L,\mathscr O_0,\mathscr E_0)}$ does not involve the choice of a parabolic subgroup with Levi complement $\mf L$, this shows that $\ind_{\mf L\subseteq\mf P}^{\mf G}(A_0)$ is independent of $\mf P$, and we can just write $\ind_{\mf L}^{\mf G}(A_0)$ instead. Next, as the elements of $\mathscr M_\mf G$ are taken up to $\mf G$-conjugacy, we may assume that $\mf L$ is a standard Levi subgroup of some standard parabolic subgroup $\mf P$ of $\mf G$ (compare \Cref{ParHC}), that is, there exists some $J\subseteq\mf S$ such that $\mf P=\mf P_J=\mf B\mf W_J\mf B$, and $\mf L=\mf L_J$ is the unique Levi complement of $\mf P_J$ containing $\mf T$.
\end{Rem}

\section{The scalars \texorpdfstring{$\xi_x$}{xix} for cuspidal \texorpdfstring{$A_x$}{Ax} in \texorpdfstring{$E_7$}{E7}, \texorpdfstring{$p=2$}{p=2}}
The notation and assumptions are as in \Cref{ParHC}, and in addition we set $p:=2$. In particular, $\mf G$ has trivial centre and we can apply \Cref{Shoji}. So there are scalars $\xi_x\in\overline{\mathbb Q}_\ell$ such that
\begin{equation}\label{scalars}
R_x=\xi_x\chi_{A_x}\quad\text{for }x\in\overline X(\mf W).
\end{equation}
Since $\langle R_x,R_x\rangle=1=\langle\chi_A,\chi_A\rangle$ for any $x\in\overline X(\mf W)$, $A\in{\hat{\mf G}}^\mathrm{un}$, we know that $\xi_x\overline\xi_x=1$ for any $x\in\overline X(\mf W)$. In this section, we determine the scalars $\xi_x$ for the cuspidal (unipotent) character sheaves $A_x$, after having chosen specific isomorphisms $F^\ast A_x\xrightarrow{\sim}A_x$ for the corresponding $x\in\overline X(\mf W)$.
\begin{Empty}\label{valcuspchar}
By \cite[Proposition 20.3]{LuCS4} there are two cuspidal character sheaves $A_1, A_2$ for $\mf G$, and both of them lie in ${\hat{\mf G}}^\mathrm{un}$. Their support is the unipotent variety $\mf G_{\mathrm{uni}}\subseteq\mf G$ consisting of all unipotent elements in $\mf G$. This variety is the (Zariski) closure of the regular unipotent conjugacy class $\mathscr O_{\mathrm{reg}}$, which is the unique class of all unipotent $x\in\mf G$ with the property $\dim C_{\mf G}(x)=\rank\mf G=7$. In particular, $\mathscr O_{\mathrm{reg}}$ is $F$-stable. Denote by $u_i=u_{\alpha_i}$ ($1\leqslant i\leqslant 7$) the closed embedding $\mf k^+\rightarrow\mf G$ whose image is the root subgroup $\mf U_{\alpha_i}\subseteq\mf U$. We set
\begin{equation}\label{u0}
u_{0}:=u_1(1)\cdot u_2(1)\cdot u_3(1)\cdot u_4(1)\cdot u_5(1)\cdot u_6(1)\cdot u_7(1)\in\mathscr O_\mathrm{reg}^F.
\end{equation}
Let $a_0$ be the image of $u_0$ in $A(u_{0})=C_{\mf G}(u_{0})/C_{\mf G}^\circ(u_{0})$ under the natural map $C_{\mf G}(u_{0})\rightarrow A(u_{0})$. Then $A(u_{0})$ is a cyclic group of order $4$ generated by $a_0$ (see \cite[\S 4]{PW} and \cite[14.15, 14.18]{DM}), and the automorphism of $A(u_{0})$ induced by $F$ is the identity. Hence, the elements of $A(u_0)$ correspond to the $\mf G^F$-conjugacy classes contained in $\mathscr O_\mathrm{reg}^F$ (see, for instance, \cite[4.3.6]{Gintro}). In particular, there are $4$ such classes and every irreducible $\mf G$-equivariant $\overline{\mathbb Q}_\ell$-local system on $\mathscr O_\mathrm{reg}$ is one-dimensional. For $a\in A(u_0)$, we choose some $x\in\mf G$ such that $x^{-1}F(x)\in C_\mf G(u_0)$ has image $a\in A(u_0)$. Then $u_a:=xu_0x^{-1}\in\mathscr O_\mathrm{reg}^F$ (defined only up to $\mf G^F$-conjugacy) represents the $\mf G^F$-conjugacy class corresponding to $a$. If $a$ is the trivial element of $A(u_0)$, we will always assume that the representative for the corresponding $\mf G^F$-conjugacy class is $u_0$. Thus $\{u_{a_0^i}\mid0\leqslant i\leqslant3\}$ is a set of representatives for the $\mf G^F$-conjugacy classes inside $\mathscr O_\mathrm{reg}^F$. Singling out the element $u_0\in\mathscr O_{\mathrm{reg}}^F$ as above (or, more precisely, the $\mf G^F$-conjugacy class of $u_0$) gives rise to specific isomorphisms $\varphi_{A_i}\colon F^\ast A_i\xrightarrow{\sim}A_i$ with respect to the two cuspidal character sheaves $A_i\in\hat{\mf G}^{\mathrm{un}}$ ($i=1,2$), as follows. Let $\mathscr E_i$ be the $F$-stable irreducible $\mf G$-equivariant $\overline{\mathbb Q}_\ell$-local system on $\mathscr O_{\mathrm{reg}}$ such that $A_i\cong\IC(\mf G_{\mathrm{uni}},\mathscr E_i)[\dim\mathscr O_\mathrm{reg}]$. By \cite[24.1]{LuCS5}, we have
\[\mathscr H^s(A_i)|_{\mf G_\mathrm{uni}}\cong\begin{cases}\mathscr E_i\;&\text{if }s=-d\\\hfil0\;&\text{if }s\neq-d,\end{cases}\]
where $d=\dim\mathscr O_\mathrm{reg}$. Hence, for $u\in\mathscr O_\mathrm{reg}^F$, the stalk $\mathscr H^{-d}_u(A_i)\cong\mathscr E_{i,u}$ is one-dimensional. Using \cite[25.1]{LuCS5}, the $\varphi_{A_i}$ may thus be chosen in such a way that, in particular, for any $u\in\mathscr O_{\mathrm{reg}}^F$, the induced map $\mathscr E_{i,u}\rightarrow\mathscr E_{i,u}$ is given by scalar multiplication with $\sqrt q^7$ times a root of unity. We can therefore modify the $\varphi_{A_i}$ by a root of unity, if necessary, such that the induced map on the stalk $\mathscr E_{i,u_0}$ is just scalar multiplication with $\sqrt q^7$, and this uniquely specifies our choices for $\varphi_{A_i}\colon F^\ast A_i\xrightarrow{\sim}A_i$ ($i=1,2$). Denoting by $\varphi_{\mathscr E_i}\colon F^\ast\mathscr E_i\xrightarrow{\sim}\mathscr E_i$ the restriction of $\varphi_{A_i}$, we then simply have
\[\chi_{A_i,\varphi_{A_i}}(u)=\chi_{\mathscr E_i,\varphi_{\mathscr E_i}}(u)\quad\text{for }u\in\mathscr O_\mathrm{reg}^F,\;i=1,2.\]   
The latter function can be computed explicitly, as described in \cite[19.7]{LuCSDC4}. Namely, assigning to a local system $\mathscr E$ on $\mathscr O_\mathrm{reg}$ its stalk $\mathscr E_{u_0}$ at $u_0$ gives a bijection between the isomorphism classes of $F$-stable $\mf G$-equivariant irreducible local systems $\mathscr E$ on $\mathscr O_\mathrm{reg}$ and $F$-invariant simple $\overline{\mathbb Q}_\ell[A(u_0)]$-modules. For $i=1,2$, let $\sigma_i\in\Irr(A(u_0))^F$ be the character corresponding to the local system $\mathscr E_i$ on $\mathscr O_\mathrm{reg}$. The $\sigma_i$ are the two faithful irreducible characters of $A(u_0)$. We number the cuspidal character sheaves $A_i$ (and, correspondingly, the $\mathscr E_i$, $\sigma_i$) in such a way that $\sigma_1(a_0)=\zeta$, $\sigma_2(a_0)=-\zeta$, where $\zeta$ is the primitive fourth root of unity in $\overline{\mathbb Q}_\ell$ chosen in \Cref{Parcompatible}. Using the cleanness results in \cite{LuclCS}, $\chi_{A_i,\varphi_{A_i}}$ vanishes outside $\mathscr O_\mathrm{reg}^F$. Hence, for $g\in\mf G^F$, we obtain
\begin{center}
\begin{tabular}{|c|c|c|c|c|c|}
\hline
& $g\notin\mathscr O_\mathrm{reg}^F$ & $g=u_{0}$ & $g=u_{a_0}$ & $g=u_{a_0^2}$ & $g=u_{a_0^3}$ \\
\hline
$\chi_{A_1,\varphi_{A_1}}(g)$ & $0$ & $q^{7/2}$ & $\phantom{-}q^{7/2}\zeta$ & $-q^{7/2}$ & $-q^{7/2}\zeta$ \\
\hline
$\chi_{A_2,\varphi_{A_2}}(g)$ & $0$ & $q^{7/2}$ & $-q^{7/2}\zeta$ & $-q^{7/2}$ & $\phantom{-}q^{7/2}\zeta$ \\
\hline
\end{tabular}
\end{center}
where $q^{7/2}=\sqrt q^7$. We can now formulate the result.
\end{Empty}
\begin{Prop}\label{scalar1}
In the setting of \Cref{valcuspchar}, let $x_1,x_2$ be the elements of $\overline X(\mf W)$ such that $A_i=A_{x_i}$, $i=1,2$. ($A_{x_i}$ is given by the parametrisation \eqref{ParUchSh}, which is uniquely determined by the requirements in \Cref{Parcompatible}.) Then we have
\[R_{x_i}=\chi_{A_i,\varphi_{A_i}}\quad\text{for }i=1,2.\]
In other words, with the choices for $\varphi_{A_i}\colon F^\ast A_i\xrightarrow{\sim}A_i$ made in \Cref{valcuspchar} ($i=1,2$), the scalars $\xi_{x_1}$, $\xi_{x_2}$ in \eqref{scalars} are both $1$.
\end{Prop}
The proof will be given in \ref{scalarpm1}--\ref{determxi} below. We start with the following simple observation.
\begin{Lm}\label{conjugate}
Consider the element $u_0=u_1(1)\cdot u_2(1)\cdot\ldots\cdot u_7(1)\in\mathscr O_{\mathrm{reg}}^F$ defined in \eqref{u0}. For any permutation $\pi$ of $\{1,2,\ldots,7\}$, there is an element $u\in\mf U^F$ such that
\[u\cdot u_0\cdot u^{-1}=u_{\pi(1)}(1)\cdot u_{\pi(2)}(1)\cdot\ldots\cdot u_{\pi(7)}(1).\]
In particular, $u_0$ is conjugate to $u_0^{-1}$ in $\mf U^F\subseteq\mf G^F$.
\end{Lm}
\begin{proof}
First consider the Weyl group $\mf W$ of $\mf G$, viewed as the irreducible finite Coxeter group of type $E_7$ with simple reflections $\mf S=\{s_1, s_2, \ldots, s_7\}$ (see \Cref{ParHC}). It is well-known that any two Coxeter elements of $(\mf W, \mf S)$ are conjugate in $\mf W$, that is, for any permutation $\pi$ of $\{1, 2,\ldots, 7\}$ there exists some $w\in\mf W$ such that
\[w\cdot(s_1\cdot s_2\cdot\ldots\cdot s_7)\cdot w^{-1}=s_{\pi(1)}\cdot s_{\pi(2)}\cdot\ldots\cdot s_{\pi(7)}.\]
More precisely, \cite[\S 1]{casscox} provides an algorithm to compute such an element $w$ which is only based on the facts that an element $s_i$ ($1\leqslant i\leqslant 7$) in the first (respectively, last) position of a given Coxeter word can be shifted to the last (respectively, first) position by means of conjugation with $s_i$, and that $s_i$, $s_j$ ($1\leqslant i, j\leqslant 7$, $i\neq j$) commute if and only if they are not linked in the Coxeter graph of $(\mf W, \mf S)$. But the analogous statements hold for the $u_i(1)$, $1\leqslant i\leqslant 7$, so we can just mimic the proof of \cite[Lemma 1.4]{casscox} to obtain an element $u\in\mf U^F$ (a product of certain $u_i(1)$, $1\leqslant i\leqslant 7$; note that $u_i(1)=u_i(-1)$ since $k$ has characteristic $2$) such that
\[u\cdot u_0\cdot u^{-1}=u_{\pi(1)}(1)\cdot u_{\pi(2)}(1)\cdot\ldots\cdot u_{\pi(7)}(1).\]
In particular, since $u_0^{-1}=u_7(1)\cdot u_6(1)\cdot\ldots\cdot u_1(1)$, $u_0$ is conjugate to $u_0^{-1}$ in $\mf U^F\subseteq\mf G^F$.
\end{proof}

\begin{Empty}\label{scalarpm1}
Let us denote by $\rho_x\in\Uch(\mf G^F)$ the unipotent character corresponding to $x\in\overline X(\mf W)$ in \eqref{ParUch}. The discussion in \Cref{Parcompatible} together with \cite[3.8]{Sh1} (or \cite[3.6]{LuRCS}) shows that $\rho_{x_1}=E_7[\zeta]$, $\rho_{x_2}=E_7[-\zeta]$ are the two cuspidal unipotent characters in Carter's table \cite[pp.~482-483]{C}. We have
\[R_{x_1}=\sum_{x\in\overline X(\mf W)}\{x,x_1\}\Delta(x)\rho_x=-\frac12\rho_{x_1}+\frac12\rho_{x_2}+\sum_{x\in\overline X(\mf W)\backslash\{x_1,x_2\}}\{x,x_1\}\rho_x,\]
\[R_{x_2}=\sum_{x\in\overline X(\mf W)}\{x,x_2\}\Delta(x)\rho_x=-\frac12\rho_{x_2}+\frac12\rho_{x_1}+\sum_{x\in\overline X(\mf W)\backslash\{x_1,x_2\}}\{x,x_2\}\rho_x.\]
(Note that $\Delta(x_1)=\Delta(x_2)=-1$, see \cite[4.14]{Luchars}.) Now consider the character $\overline\rho_{x_1}$, defined by
\[\overline\rho_{x_1}(g):=\overline{\rho_{x_1}(g)}\quad\text{for }g\in\mf G^F.\]
$\overline\rho_{x_1}$ is again a cuspidal unipotent (irreducible) character of $\mf G^F$. Since the character field of the two cuspidal unipotent characters of groups of type $E_7$ contains non-real elements (see \cite[5.4, Table 1]{GSchur}), we conclude $\overline\rho_{x_1}=\rho_{x_2}$. In view of \Cref{conjugate} and since $\{x,x_1\}=\{x,x_2\}$ for any $x\in\overline X(\mf W)\backslash\{x_1, x_2\}$, we obtain
\[\xi_{x_1}q^{7/2}=\xi_{x_1}\chi_1(u_{0})=R_{x_1}(u_{0})=R_{x_2}(u_{0})=\xi_{x_2}\chi_2(u_{0})=\xi_{x_2}q^{7/2},\]
which also equals $\overline{R_{x_2}(u_{0})}$ (using once again \Cref{conjugate}). We deduce $\xi_{x_1}=\xi_{x_2}=\overline\xi_{x_2}$ and thus $\xi_{x_1}=\xi_{x_2}\in\{\pm1\}$, since $\xi_{x_2}\overline\xi_{x_2}=1$.
\end{Empty}

\begin{Empty}\label{scalarfixsign}
In order to determine the sign $\xi:=\xi_{x_1}=\xi_{x_2}\in\{\pm1\}$, we want to apply a formula in \cite[3.6(b)]{GCH}, as in \cite{HE6p3}. However, we will need more theoretical ingredients than in \cite{HE6p3}, most notably the generalised Springer correspondence (see \Cref{GenSpring}). We consider the Hecke algebra of the group $\mf G^F$ with its $BN$-pair $(\mf B^F,N_{\mf G}(\mf T)^F)$, that is, the endomorphism algebra
\[\mathcal H:=\End_{\overline{\mathbb Q}_\ell\mf G^F}\bigl(\overline{\mathbb Q}_\ell[{\mf G^F}/{\mf B^F}]\bigr)^\mathrm{opp}\]
(\enquote{opp} stands for the opposite algebra). $\mathcal H$ has a $\overline{\mathbb Q}_\ell$-basis $\{T_w\mid w\in\mf W\}$, where
\[T_w\colon\overline{\mathbb Q}_\ell[{\mf G^F}/{\mf B^F}]\rightarrow\overline{\mathbb Q}_\ell[{{\mf G}^F}/{\mf B^F}],\quad x\mf B^F\mapsto\sum_{\substack{y\mf B^F\in{\mf G^F}/{\mf B^F} \\ x^{-1}y\in\mf B^F\dot w \mf B^F}}y\mf B^F,\]
for $w\in\mf W$. Recall from \ref{ParHC} that $\mf S=\{s_1,\ldots,s_7\}$ is a set of Coxeter generators for $\mf W$. Denote by $\ell\colon\mf W\rightarrow\mathbb Z_{\geqslant0}$ the length function of $\mf W$ with respect to $\mf S$. Then the multiplication in $\mathcal H$ is determined by the following equations.
\[T_s\cdot T_w=\begin{cases}\hfil T_{sw}\quad&\text{if}\quad\ell(sw)=\ell(w)+1 \\ qT_{sw}+(q-1)T_{w}\quad&\text{if}\quad\ell(sw)=\ell(w)-1\end{cases}\quad(\text{for }s\in\mf S,\, w\in\mf W).\]
The irreducible characters of $\mf W$ naturally parametrise the isomorphism classes of irreducible modules of $\mathcal H$, see \cite[\S 11D]{CR}. Given $\phi\in\Irr(\mf W)$, let $V_\phi$ be the corresponding module of $\mathcal H$, and let $\rho_\phi\in\Uch(\mf G^F)$ be the image of $\phi$ under the map $\Irr(\mf W)\hookrightarrow\mathfrak S_\mf W\xrightarrow{\sim}\Uch(\mf G^F)$, see \Cref{Parcompatible}. By \cite[3.6]{GCH} and \cite[\S 8.4]{GePf}, we have
\begin{equation}\label{HeckeUch} 
\sum_{\phi\in\Irr(\mf W)}\rho_\phi(g)\Tr(T_w,V_\phi)=\frac{|O_g\cap\mf B^F\dot w\mf B^F|\cdot|C_{\mf G^F}(g)|}{|\mf B^F|}
\end{equation}
for any $g\in\mf G^F$ and $w\in\mf W$, where $\dot w\in N_{\mf G}(\mf T)^F$ is a representative of $w$, $O_g\subseteq\mf G^F$ denotes the $\mf G^F$-conjugacy class of $g$ and $\Tr(T_w,V_\phi)$ is the trace of the linear map on $V_\phi$ defined by $T_w$. The character table of $\mathcal H$ is contained in {\sffamily {CHEVIE}} \cite{CHEVIE}, so the numbers $\Tr(T_w,V_\phi)$ are known. We want to apply \eqref{HeckeUch} with $g=u_{0}$ to get hold of the sign $\xi$. In order to achieve this, we thus need some information about the values $\rho_\phi(u_0)$ for $\phi\in\Irr(\mf W)$. By \eqref{Fourier} and \eqref{scalars}, we have
\[\rho_\phi=\Delta(x_\phi)\sum_{x\in\overline X(\mf W)}{\{x_\phi,x\}}R_x=\sum_{x\in\overline X(\mf W)}{\{x_\phi,x\}}\xi_x\chi_{A_x}.\]
\end{Empty}
\begin{Lm}\label{vanishingAx}
Let $x\in\overline X(\mf W)$ be such that, via the generalised Springer correspondence \eqref{SpringerF}, $A_x$ arises from a pair $\iota=(\mathscr O,\mathscr E)\in\mathscr N_\mf G^F$ where $\mathscr O\neq\mathscr O_\mathrm{reg}$. Then
\[\chi_{A_x}(u_0)=R_x(u_0)=0.\]
\end{Lm}
\begin{proof}
By \eqref{CSIC}, the support of $A_x|_{\mf G_\mathrm{uni}}$ is $\overline{\mathscr O}$ which is a union of conjugacy classes of $\mf G$ and therefore does not contain any regular unipotent elements in case $\mathscr O\neq\mathscr O_{\mathrm{reg}}$, as $\mathscr O_{\mathrm{reg}}$ is the unique class of $\mf G$ whose closure is $\mf G_{\mathrm{uni}}$. So we have $\chi_{A_x}(u_0)=0$ and then also $R_x(u_0)=0$ in view of \Cref{Shoji}.
\end{proof}
\begin{Empty}\label{GenSpringerOreg}
Let us now describe the images of the pairs $\iota=(\mathscr O_\mathrm{reg},\mathscr E)\in\mathscr N_\mf G^F$ under the map \eqref{SpringerF}, following Spaltenstein \cite[p.~331]{Sp}. Since $F$ acts trivially on $A(u_0)\cong{\mathbb Z}/{4\mathbb Z}$, there are (up to isomorphism) four $F$-stable irreducible $\mf G$-equivariant local systems on the regular unipotent class $\mathscr O_\mathrm{reg}$. We will write $\mathscr E^{\zeta^i}$ ($1\leqslant i\leqslant4$) for the local system corresponding to the irreducible character of $A(u_0)$ which takes the value $\zeta^i$ at $a_0=u_0C_{\mf G^\circ}(u_0)\in A(u_0)$, where $\zeta$ is the primitive fourth root of unity in $\overline{\mathbb Q}_\ell$ specified in \ref{valcuspchar}, \ref{Parcompatible}. First consider $\iota=(\mathscr O_{\mathrm{reg}},\mathscr E^1)\in\mathscr N_\mf G^F$. The corresponding element of $\mathscr M_\mf G^F$ is $(\mf T,\{1\},\overline{\mathbb  Q}_\ell)$, and among the simple summands of $K_{(\mf T,\{1\},\overline{\mathbb Q}_\ell)}$, the character sheaf $A_x\cong A_{(\mathscr O_\mathrm{reg},1)}$ is parametrised by the trivial character $1_\mf W$ of $\mf W=W_\mf G(\mf T)$, whence $R_x=R_{1_\mf W}=\rho_{1_\mf W}=1_{\mf G^F}$. Next, the pairs $(\mathscr O_\mathrm{reg},\mathscr E^{\pm\zeta})$ correspond to $(\mf G,\mathscr O_\mathrm{reg},\pm\zeta)$, so they give rise to the two cuspidal (unipotent) character sheaves $A_{(\mathscr O_\mathrm{reg},\pm\zeta)}$. With the arrangements in \ref{valcuspchar}, \ref{scalar1}, we have $A_{x_1}\cong A_{(\mathscr O_\mathrm{reg},\zeta)}$ and $A_{x_2}\cong A_{(\mathscr O_\mathrm{reg},-\zeta)}$. Finally, to the pair $(\mathscr O_\mathrm{reg},\mathscr E^{-1})\in\mathscr N_\mf G^F$ is associated the triple $(\mf L,\mathscr O_0,\mathscr E_0)\in\mathscr M_\mf G^F$ where $\mf L\subseteq\mf G$ is the Levi complement of some parabolic subgroup of $\mf G$ such that ${\mf L}/{{\mf Z(\mf L)}^\circ}$ is simple of type $D_4$, and $(\mathscr O_0,\mathscr E_0)$ is the unique cuspidal pair for $\mf L$, see \cite[15.2]{LuIC}. Among the simple summands of $K_{(\mf L,\mathscr O_0,\mathscr E_0)}$, the character sheaf $A_{(\mathscr O_\mathrm{reg},\mathscr E^{-1})}$ is described by the trivial character of $\mf W_{\mf G}(\mf L)\cong W(B_3)$. Let $x_0\in\overline X(\mf W)$ be such that $A_{x_0}\cong A_{(\mathscr O_\mathrm{reg},\mathscr E^{-1})}$, and let $\mathscr F\subseteq\Irr(\mf W)$ be the family for which $x_0\in\mathscr M(\mathscr G_\mathscr F)$. Then, for $\phi\in\mathscr F$, we have
\[\rho_\phi(u_0)=\sum_{x\in\mathscr M(\mathscr G_\mathscr F)}\{x_\phi,x\}R_x(u_0)=\{x_\phi,x_0\}R_{x_0}(u_0).\]
Using the notation of \cite[p.~482]{C}, we obtain
\[\rho_{\phi_{56,3}}(u_0)=-\rho_{\phi_{35,4}}(u_0)=-\rho_{\phi_{21,6}}(u_0)=\frac12 R_{x_0}(u_0).\]
Now, considering the two characters $\phi_{512,11}$, $\phi_{512,12}\in\Irr(\mf W)$ of degree $512$, we see that
\begin{align}\label{valexcchars}
\begin{split}
\rho_{\phi_{512,11}}(u_0)=\{x_{\phi_{512,11}},x_1\}\xi\chi_{A_{x_1}}(u_0)+\{x_{\phi_{512,11}},x_2\}\xi\chi_{A_{x_2}}(u_0)=\phantom{-}\xi q^{7/2}, \\
\rho_{\phi_{512,12}}(u_0)=\{x_{\phi_{512,12}},x_1\}\xi\chi_{A_{x_1}}(u_0)+\{x_{\phi_{512,12}},x_2\}\xi\chi_{A_{x_2}}(u_0)=-\xi q^{7/2}.
\end{split}
\end{align}
For any $\phi\in\Irr(\mf W)$ which is not one of $1_\mf W=\phi_{1,0}$, $\phi_{56,3}$, $\phi_{35,4}$, $\phi_{21,6}$, $\phi_{512,11}$, $\phi_{512,12}$, we have $\rho_\phi(u_0)=0$. We evaluate the left side of \eqref{HeckeUch} with $g=u_0$ and
\[w=w_\mathrm c:=s_1\cdot s_2\cdot s_3\cdot s_4\cdot s_5\cdot s_6\cdot s_7\in\mf W,\]
a Coxeter element of $\mf W$. (In fact any other Coxeter element of $\mf W$ would lead to the same result, see \cite[8.2.6]{GePf}.) We obtain
\[q^7+\frac12R_{x_0}(u_0)\left(\Tr(T_{w_\mathrm c},V_{\phi_{56,3}})-\Tr(T_{w_\mathrm c},V_{\phi_{35,4}})-\Tr(T_{w_\mathrm c},V_{\phi_{21,6}})\right)+2\xi q^7.\]
Here, $\Tr(T_{w_\mathrm c},V_{\phi_{56,3}})-\Tr(T_{w_\mathrm c},V_{\phi_{35,4}})-\Tr(T_{w_\mathrm c},V_{\phi_{21,6}})=2q^5$, so we have
\begin{equation}\label{sumu0wc}
\sum_{\phi\in\Irr(\mf W)}\rho_\phi(u_0)\Tr(T_{w_\mathrm c},V_\phi)=q^7(1+2\xi)+q^5R_{x_0}(u_0).
\end{equation}
\end{Empty}

\begin{Lm}\label{Rx0u0}
As in \Cref{GenSpringerOreg} (and with the notations there), let $x_0\in\overline X(\mf W)$ be such that $A_{x_0}\cong A_{(\mathscr O_{\mathrm{reg}},\mathscr E^{-1})}$. Then there is a sign $\delta\in\{\pm1\}$ such that
\[R_{x_0}(u_0)=\delta q^2.\]
\end{Lm}
\begin{proof}
We keep the notation of \Cref{GenSpringerOreg} and set $\iota:=(\mathscr O_{\mathrm{reg}},\mathscr E^{-1})\in\mathscr N_{\mf G}^F$. The character sheaf $A_{x_0}\cong A_{\iota}\in\hat{\mf G}^{\mathrm{un}}$ is an irreducible direct summand of $K:=K_{(\mf L,\mathscr O_0,\mathscr E_0)}$, where $\mf L\subseteq\mf G$ is the Levi complement of some parabolic subgroup of $\mf G$ such that ${\mf L}/{{\mf Z(\mf L)}^\circ}$ is simple of type $D_4$, and $(\mathscr O_0,\mathscr E_0)$ is the unique cuspidal pair for $\mf L$. Among the simple summands of $K$, $A_{\iota}$ is uniquely characterised by the property
\[A_{\iota}|_{\mf G_{\mathrm{uni}}}\cong\IC(\mf G_{\mathrm{uni}},\mathscr E^{-1})[\dim{\mathscr O_{\mathrm{reg}}}+\dim{{\mf Z(\mf L)}^\circ}],\]
see \eqref{CSIC}. Now consider the functions $X_{\iota}$, $Y_{\iota}$ defined in \Cref{PreAlgF}. Using \Cref{Alg} and \Cref{RemarkAlg}, we see that
\[\chi_{A_\iota,\varphi_{A_\iota}}(u_0)=-X_{\iota}(u_0)=-Y_{\iota}(u_0)=-\Trace(\psi_{\iota,u_0},\mathscr E^{-1}_{u_0}).\]
In order to meet the requirements of \Cref{LusztigONB} (see \cite[25.1]{LuCS5}), we need to modify the isomorphism $\varphi_{A_\iota}\colon F^\ast A_\iota\xrightarrow{\sim}A_\iota$, as follows. Let $\Sigma_0:={\mf Z(\mf L)}^\circ\mathscr O_0$. By \cite[19.3]{LuCS4}, the image of $\Sigma_0$ (or $\mathscr O_0$) under the canonical map ${\mf L}\rightarrow{\mf L}/{{\mf Z(\mf L)}^\circ}$ is the regular unipotent class of ${\mf L}/{{\mf Z(\mf L)}^\circ}$. As in \cite[4.4]{Gvaluni}, we may single out a specific element $g_1\in(\Sigma_0/{\mf Z(\mf L)}^\circ)^F$. Following \cite[3.2]{Luvaluni}, the choice of $g_1$ then gives rise to an isomorphism $\varphi_0\colon F^\ast\mathscr E_0\xrightarrow{\sim}\mathscr E_0$ by requiring that it induces on the stalk of $\mathscr E_0$ at $g_1$ (or, more precisely, at the preimage of $g_1$ under $\mf L\rightarrow{\mf L/{{\mf Z(\mf L)}^\circ}}$) the map given by scalar multiplication with $q^{({\dim{({\mf L}/{{\mf Z(\mf L)}^\circ})}-\dim{\mathscr O_0}})/2}=q^2$. Once $\varphi_0$ is fixed, it naturally induces an isomorphism $F^\ast K\xrightarrow{\sim}K$, which in turn determines an isomorphism $\varphi_{A_{x_0}}\colon{F^\ast A_{x_0}}\xrightarrow{\sim}{{A_{x_0}}}$, see again \cite[3.2]{Luvaluni}. Then $\varphi_{A_{x_0}}$ satisfies the requirements formulated in \cite[25.1]{LuCS5}, so it is a valid choice for \Cref{LusztigONB}. In particular, we have $\langle\chi_{A_{x_0},\varphi_{A_{x_0}}},\chi_{A_{x_0},\varphi_{A_{x_0}}}\rangle=1$. On the other hand, recall from \Cref{PreAlgF} that the definition of $\varphi_{A_\iota}\colon F^\ast A_\iota\xrightarrow{\sim}A_\iota$ is based on the choice of an isomorphism $F^\ast\mathscr E_0\xrightarrow{\sim}\mathscr E_0$ which induces a map of finite order on the stalk of $\mathscr E_0$ at any element of $\mathscr O_0^F$. Modifying this isomorphism by multiplication with a root of unity, if necessary, we may thus achieve that $\varphi_{A_{x_0}}=q^2\cdot\varphi_{A_\iota}$. Now, by \Cref{Shoji}, we have 
\[R_{x_0}=\xi_{x_0}\chi_{A_{x_0},\varphi_{A_{x_0}}}\]
for some $\xi_{x_0}\in\overline{\mathbb Q}_\ell$ such that $\xi_{x_0}\overline\xi_{x_0}=1$. So we obtain
\begin{align*}
R_{x_0}(u_0)&=\xi_{x_0}\chi_{A_{x_0},\varphi_{A_{x_0}}}(u_0)=\xi_{x_0}q^2\cdot\chi_{A_\iota,\varphi_{A_\iota}}(u_0)=-\xi_{x_0}q^2\cdot\Trace(\psi_{\iota,u_0},\mathscr E^{-1}_{u_0}) \\
&=\delta q^2,
\end{align*}
where $\delta:=-\xi_{x_0}\Trace(\psi_{\iota,u_0},\mathscr E^{-1}_{u_0})\in\overline{\mathbb Q}_\ell$. Since the local system $\mathscr E^{-1}$ is one-dimensional and $\psi_{\iota,u_0}\colon\mathscr E^{-1}_{u_0}\rightarrow\mathscr E^{-1}_{u_0}$ has finite order, the number $\Trace(\psi_{\iota,u_0},\mathscr E^{-1}_{u_0})$ is a root of unity in $\overline{\mathbb Q}_\ell$, so $\delta$ satisfies $\delta\overline\delta=1$, that is, $\overline\delta=\delta^{-1}$. Using \Cref{conjugate}, we conclude
\[\delta q^2=R_{x_0}(u_0)=R_{x_0}(u_0^{-1})=\overline\delta q^2,\]
so we must have $\delta=\overline\delta=\delta^{-1}\in\{\pm1\}$, as desired.
\end{proof}

\begin{Empty}\label{determxi}
In view of \Cref{Rx0u0}, \eqref{HeckeUch} and \eqref{sumu0wc}, we have
\begin{equation}\label{xidelta}
\frac{|O_{u_0}\cap\mf B^F\dot w_\mathrm c\mf B^F|\cdot|C_{\mf G^F}(u_0)|}{|\mf B^F|}=q^7(1+2\xi+\delta),
\end{equation}
where $\xi\in\{\pm1\}$ is as in \Cref{scalarfixsign}, $\delta\in\{\pm1\}$ is as in \Cref{Rx0u0}, and $\dot w_\mathrm c\in N_{\mf G}(\mf T)^F$ is a representative of $w_\mathrm c\in\mf W$. We claim that $O_{u_0}\cap\mf B^F\dot w_\mathrm c\mf B^F\neq\emptyset$. Indeed, with the notation of \Cref{ParHC}, consider the longest element $w_0$ of $\mf W$ with respect to the usual length function on $\mf W$ determined by the simple roots $\Pi=\{\alpha_1,\ldots,\alpha_7\}\subseteq R^+$. $\mf W$ is a Coxeter group with generators $\{s_1,\ldots,s_7\}$, and there is a natural action of $\mf W$ on the roots $R$. Using some standard properties of finite Coxeter groups (see, e.g., \cite[Appendix]{StChev}), the element $w_0$ is then characterised by the property $w_0(R^+)=-R^+$, so we have $-w_0(\Pi)=\Pi$. More precisely, $-w_0$ defines a graph automorphism of the Dynkin diagram of $E_7$, but the only such automorphism is the identity, so $-w_0(\alpha_i)=\alpha_i$ for $1\leqslant i\leqslant7$. Let us choose a representative $\dot w_0\in N_{\mf G}(\mf T)^F$ of $w_0\in\mf W$. For $1\leqslant i\leqslant 7$, we have $\dot w_0\mf U_{\alpha_i}\dot w_0^{-1}=\mf U_{-\alpha_i}$, so $\dot w_0 u_i(1)\dot w_0^{-1}\in\mf U_{-\alpha_i}$. Now 
\[\mf U_{-\alpha_i}\subseteq\mf L_{\{s_i\}}\subseteq\mf P_{\{s_i\}}=\mf B\cup\mf B s_i\mf B\]
(see, e.g., \cite[2.6.2]{C}). Since $\mf U_{-\alpha_i}\cap\mf B=\{1\}$, we must have $\dot w_0 u_i(1)\dot w_0^{-1}\in\mf B s_i\mf B$. As $w_\mathrm c=s_1\cdot\ldots\cdot s_7$ is a reduced expression for $w_\mathrm c$, we know that
\[\mf Bs_1\mf B\cdot\mf Bs_2\mf B\cdot\ldots\cdot\mf Bs_7\mf B=\mf Bw_\mathrm c\mf B,\]
and so
\[\dot w_0\cdot u_0\cdot\dot w_0^{-1}\in O_{u_0}\cap\mf Bw_\mathrm c\mf B=O_{u_0}\cap\mf B^F\dot w_\mathrm c\mf B^F.\]
Hence, the left side of \eqref{xidelta} is strictly positive. Thus \eqref{xidelta} implies that $\xi\neq-1$, so we must have $\xi=+1$. This proves \Cref{scalar1}.
\end{Empty}

\begin{Rem}
As described in \Cref{valcuspchar}, the isomorphisms $\varphi_{A_i}\colon F^\ast A_i\xrightarrow{\sim}A_i$ ($i=1,2$) are determined by choosing one of the four $\mf G^F$-conjugacy classes inside $\mathscr O_{\mathrm{reg}}^F$. The proof of \Cref{scalar1} relies on the following two properties of the $\mf G^F$-conjugacy class $O_{u_0}$ of $u_0$:
\begin{enumerate}
\item Any element of $O_{u_0}$ is $\mf G^F$-conjugate to its inverse, that is, $O_{u_0}=O_{u_0}^{-1}$;
\item We have $O_{u_0}\cap\mf B^F\dot w_\mathrm c\mf B^F\neq\emptyset$.
\end{enumerate}
Once $\xi$ is specified, a similar computation as in \Cref{determxi} shows that $O_{u_{a_0^2}}\cap\mf B^F\dot w_\mathrm c\mf B^F=\emptyset$. Hence, among the four $\mf G^F$-classes inside $\mathscr O_{\mathrm{reg}}^F$, the class $O_{u_0}$ is uniquely determined by (1) and (2) above. 
\end{Rem}

\begin{Empty}
We conclude by giving the values of the (unipotent) characters in the family 
\[\mathscr F_0=\left\{\rho_{\phi_{512,11}}, \rho_{\phi_{512,12}}, E_7[\zeta], E_7[-\zeta]\right\}\subseteq\overline X(\mf W)\]
at the regular unipotent elements. Using \eqref{Fourier}, \ref{vanishingAx} and \ref{scalar1}, we have for $u\in\mathscr O_{\mathrm{reg}}^F$:
\begin{align*}
\rho_{\phi_{512,11}}(u)=-\rho_{\phi_{512,12}}(u)=\bigl(\chi_{A_1,\varphi_{A_1}}(u)+\chi_{A_2,\varphi_{A_2}}(u)\bigr)/2; \\
E_7[\zeta](u)=-E_7[-\zeta](u)=-\bigl(\chi_{A_1,\varphi_{A_1}}(u)-\chi_{A_2,\varphi_{A_2}}(u)\bigr)/2.
\end{align*}
The values of the $\chi_{A_i,\varphi_{A_i}}$ are given at the end of \Cref{valcuspchar}. We thus get
\begin{center}
\begin{tabular}{|c|c|c|c|c|}
\hline
& $u\sim u_{0}$ & $u\sim u_{a_0}$ & $u\sim u_{a_0^2}$ & $u\sim u_{a_0^3}$ \\
\hline
$\rho_{\phi_{512,11}}(u)$ & $\phantom{-}q^{7/2}$ & $0$ & $-q^{7/2}$ & $0$ \\
\hline
$\rho_{\phi_{512,12}}(u)$ & $-q^{7/2}$ & $0$ & $\phantom{-}q^{7/2}$ & $0$ \\
\hline
$E_7[\zeta](u)$ & $0$ & $-q^{7/2}\zeta$ & $0$ & $\phantom{-}q^{7/2}\zeta$ \\
\hline
$E_7[-\zeta](u)$ & $0$ & $\phantom{-}q^{7/2}\zeta$ & $0$ & $-q^{7/2}\zeta$ \\
\hline
\end{tabular}
\end{center}
where $u\sim u_a$ means that $u$ lies in the $\mf G^F$-conjugacy class corresponding to $a\in A(u_0)$ as described in \Cref{valcuspchar}.
\end{Empty}

\subsection*{Acknowledgements}
I thank Meinolf Geck for many comments and hints, as well as Gunter Malle for comments on an earlier version. I also wish to thank an unknown referee for detailed suggestions on how to improve the paper. This work was supported by DFG SFB-TRR 195.

\bibliographystyle{hacm}

\bibliography{E7p2-arXiv-v3}

\end{document}